\documentclass[12pt,a4paper]{article}

\usepackage{latexsym}
\usepackage{amsmath}
\usepackage{amssymb}
\usepackage{amsthm}
\usepackage{amscd}
\usepackage{mathrsfs}
\usepackage[all]{xy}
\usepackage{graphicx}
\usepackage{comment}
\usepackage{bm}

\newtheorem{theorem}{Theorem}[section]
\newtheorem{lemma}[theorem]{Lemma}
\newtheorem{corollary}[theorem]{Corollary}
\newtheorem{proposition}[theorem]{Proposition}

\theoremstyle{definition}

\newtheorem{remark}[theorem]{Remark}
\newtheorem{definition}[theorem]{Definition}

\theoremstyle{remark}

\makeatletter

\@addtoreset{equation}{section}
\makeatother

\DeclareMathOperator{\tr}{tr}

\DeclareMathOperator{\Hom}{Hom}

\DeclareMathOperator{\Tr}{Tr}
\DeclareMathOperator{\Nr}{Nr}
\DeclareMathOperator{\Trd}{Trd}
\DeclareMathOperator{\Nrd}{Nrd}

\newcommand{\bZ}{\mathbb{Z}}

\newcommand{\bC}{\mathbb{C}}

\newcommand{\cA}{\mathcal{A}}

\newcommand{\cE}{\mathcal{E}}
\newcommand{\cH}{\mathcal{H}}

\newcommand{\cO}{\mathcal{O}}

\newcommand{\fp}{\mathfrak{p}}

\newcommand{\iGL}{\mathit{GL}}

\usepackage{color}

\begin{document}
\title{Local Jacquet-Langlands correspondences\\ 
for simple supercuspidal representations} 
\author{Naoki Imai and Takahiro Tsushima}
\date{}
\maketitle


\begin{abstract} 
We give a description of the local Jacquet-Langlands 
correspondence for simple supercuspidal representations 
via type theory. 
As a consequence, we show that the 
endo-classes for such representations are invariant under 
the local Jacquet-Langlands correspondence. 
\end{abstract}
\section*{Introduction}
Let $K$ be a non-archimedean local field 
with residue characteristic $p$. 
Let $A$ be a central simple algebra over $K$. 
We put $n=[A:K]^{1/2}$. 
The local Jacquet-Langlands 
correspondence (LJLC) 
gives a correspondences between irreducible 
essentially square-integrable representations of 
$\iGL_n (K)$ and $A^{\times}$. 
The irreducible supercuspidal representations of 
$\iGL_n (K)$ is classified in \cite{BK} 
via type theory, which describes 
supercuspidal representations as compact inductions 
of representations 
of some open subgroups that are compact modulo center. 
More generally, 
type theory for representations of $A^{\times}$ 
is developed in a series of papers 
\cite{SecGLmD1}, \cite{SecGLmD2}, \cite{SecGLmD3}, 
\cite{SeSt4}, \cite{BSS5} and \cite{SeSt6}. 
So it is natural to seek 
a description of the LJLC via type theory. 

In the case where $A$ is a division algebra, 
such descriptions are studied in \cite{GeFl4}, \cite{HeJLexI}, 
\cite{BHJLexII} and \cite{BHbook} 
if $n$ is a prime number, 
and in \cite{BHtliftIII} if $p$ is odd, $n$ is a power of $p$ and 
representations are totally ramified. 
For general $A$, 
such descriptions are given 
in \cite{SZmatzero} for 
level zero discrete series representations, and 
in \cite{BHetJL} 
for essentially tame representations. 

%A notion of epipelagic representations is introduced in \cite{RYepi}. 
%In the case for $\textit{GL}_n(K)$, the epipelagic representations are 
%the irreducible supercuspidal representations with 
%Artin conductor $n+1$ according to \cite[p.~433]{BHepi}. 
%This class of representations are called 
%simple supercuspidal representations in 
%\cite{ALssrGL}, \cite{KLssr} and \cite{XussrGL}. 
%In this paper, we call such representations 
%simple supercuspidal representations as in \cite{ITepitame}. 
%Actually, the epipelagic representations mean 
%a wider class of representations in \cite{ALssrGL}. 

In this paper, we give a description of the LJLC via type theory 
for the simple supercuspidal representations. 
We define the simple supercuspidal representations of $A^{\times}$ 
in Definition \ref{def:sims} after \cite{RYepi}, 
which are equivalent to 
the supercuspidal representations of conductor $n+1$ as a result. 
Such representations appear in \cite{ITepitame}. 
If $A^{\times} =\iGL_n (K)$, 
simple supercuspidal representations are studied in 
\cite{ALssrGL} (cf.~\cite{KLssr}, \cite{XussrGL}) and 
they are called epipelagic representations in \cite{BHepi}. 
We note that the simple supercuspidal representations of $A^{\times}$ 
are essentially tame if $n$ is prime to $p$, 
but they are not essentially tame if 
$p$ divides $n$. 

As a consequence of the description of the LJLC, 
we show that 
the endo-classes are invariant under the 
LJLC for the simple supercuspidal representations. 
This verifies 
\cite[Conjecture 9.5]{BSS5} by Broussous-S\'{e}cherre-Stevens 
for the simple supercuspidal representations. 
The conjecture is verified in \cite{KaEndoJL} for 
totally ramified representations of 
unit groups of division algebras, 
based on results in \cite{BHtliftIII}. 

In Section \ref{ClassRep}, 
we give a construction and a definition of 
the simple supercuspidal representations 
of $A^{\times}$. Further, we 
show that they are equivalent to 
the supercuspidal representations of conductor $n+1$. 
In Section \ref{character}, 
we give formulas for the characters 
of the simple supercuspidal representations 
at some elements, 
which are elliptic quasi-regular in the sense of 
\cite[1.1 Remark]{BHetJL}. 
The values of the characters at these elements 
are written in terms of variants of 
Gauss sums and generalized Kloosterman sums
(cf.~Proposition \ref{trGa} and Proposition \ref{trKl}). 
We believe that these formulas are interesting in themselves. 
In Section \ref{DesLJLC}, 
we give a description of the LJLC 
via type theory in Theorem \ref{mainthm}. 
We determine the description 
by checking character relations at some elements. 
In Section \ref{An}, we give 
another proof of our main theorem 
using Godement-Jacquet local constants. 
This proof is based on 
a formula in \cite{B} and \cite{BFgauss}, which calculates Godement-Jacquet local constants
with respect to some kind of Gauss sums. 
In Section \ref{Invec}, 
we show the invariance of the endo-classes under the 
LJLC for the simple supercuspidal representations. 

Type theory for 
irreducible supercuspidal representations of 
$A^{\times}$ 
naturally appears in the study of geometric realization of 
the local Langlands correspondence and the LJLC. 
The results in this paper are used in \cite{ITepiwild} 
to show that 
the LJLC is realized in the cohomology of the reductions of 
a family of affinoids in the Lubin-Tate perfectoid space. 

\subsection*{Acknowledgements}
This work was supported by JSPS KAKENHI Grant Numbers 
26707003, 15K17506. 
We thank a referee for suggestions for improvement.

\section*{Notation}
For a non-archimedean local 
field $F$, let $\mathcal{O}_F$ 
denote the ring of integers in $F$. 
For a field $F$ and a positive integer $l$, 
let $\mu_l (F)$ denote the group of $l$-th roots of the unity inside $F$. For an abelian group $X$, 
we write $X^{\vee}$ for its character group 
$\mathrm{Hom}_{\mathbb{Z}}(X,\mathbb{C}^{\times})$. 
For a group $G$, its subgroup $H$, 
a character $\theta$ of $H$ and $g \in G$, 
we put $H^g =g H g^{-1}$ and define a character 
$\theta^g$ of $H^g$ by 
$\theta^g (h) =\theta (g^{-1} h g)$ 
for $h \in H^g$. 

\section{Simple supercuspidal representation}\label{ClassRep}
In this section, we give a construction and a definition of 
the simple supercuspidal representations of 
the multiplicative group of a central simple algebra over 
a non-archimedean local field. 
Further, we give a characterization of 
the simple supercuspidal representations by the conductors. 

Let $K$ be a non-archimedean local field 
with residue field $k$. 
We set $q=|k|$. 
Let $\mathfrak{p}_K$ be the maximal ideal 
of $\mathcal{O}_K$. 
Any central simple algebra over $K$ 
is isomorphic a matrix algebra over 
a central division algebra over $K$. 
Let $m$ be a positive integer, and 
$D$ be a central division algebra over $K$. 
We put $A=M_m(D)$ and $G=A^{\times}$. 

Let $\pi$ be an irreducible smooth representation of 
$G$. 
We fix a non-trivial character $\psi \in k^{\vee}$. 
For $x \in \cO_K$, 
let $\bar{x}$ denote the image of 
$x$ under the reduction map $\cO_K \to k$. 
We take a character 
$\psi_K \in K^{\vee}$ such that 
\begin{align*}
 \psi_K (x) &=\psi (\bar{x}) \quad \textrm{for} \ 
 x \in \mathcal{O}_K, \\ 
 \psi_K (x) &=1  \quad \textrm{for} \ 
 x \in \mathfrak{p}_K. 
\end{align*}
Let 
$\epsilon( \pi, s, \psi_K)$ be the Godement-Jacquet local constant of 
$\pi$ with respect to $\psi_K$. 
Then there exists an integer $f(\pi, \psi_K)$ such that 
\[
 \epsilon(\pi,s,\psi_K) 
 =q^{-f(\pi, \psi_K) s} 
 \epsilon(\pi, 0, \psi_K) 
\]
by \cite[Theorem 3.3(4)]{GJzf}. 
We put $r=[D:K]^{1/2}$ and $n=mr$. 
We define the conductor $c (\pi)$ of $\pi$ by 
\[
 c(\pi)=f(\pi, \psi_K)+n. 
\]
Let 
$\cA_{D,m}^{n+1}$ 
denote the set of the isomorphism classes 
of the supercuspidal representations of $G$ of conductor $n+1$.

We fix a uniformizer $\varpi$ of $K$. 
Let 
$\eta=(\zeta,\chi,c) \in \mu_{q-1}(K) \times 
 (k^{\times})^{\vee} \times \mathbb{C}^{\times}$. 
In the following, 
we define a smooth representation 
$\pi_{D,m,\eta}$ of $G$. 
Let $K_r$ be the unramified 
extension of $K$ of degree $r$. 
We take an element 
$\varphi_{D,\zeta} \in D^{\times}$, 
an embedding $K_r \hookrightarrow D$ and 
an integer $1 \leq s \leq r-1$ 
which is prime to $r$ 
such that 
$\varphi_{D,\zeta}^r=\zeta \varpi$ and 
$\varphi_{D,\zeta} d \varphi_{D,\zeta}^{-1} =d^{q^s}$ 
for 
$d \in \mu_{q^r -1}(K_r)$. 
We put 
\[
 \varphi_{\zeta}=
 \begin{pmatrix}
 \bm{0} & I_{m-1} \\
 \varphi_{D,\zeta} & \bm{0}
 \end{pmatrix} \in A 
 \quad \textrm{and} \quad 
 L_{\zeta}=K(\varphi_{\zeta}) \subset A. 
\] 
Since $\varphi_{\zeta}^n=\zeta \varpi$, the field $L_{\zeta}$ is 
a totally ramified extension over $K$ 
of degree $n$. 
We have 
\begin{equation}\label{Nrd}
\Nrd_{A/K}(\varphi_{\zeta})=(-1)^{n-1} \zeta \varpi, \quad 
\Trd_{A/K}(\varphi_{\zeta}^{-1})=0. 
\end{equation}

Let $\mathcal{O}_D$ denote the maximal 
order of $D$, 
and $\mathfrak{p}_D$ denote the maximal 
ideal of $\mathcal{O}_D$. 
For a positive integer $l$, 
let $k_l$ be the extension of $k$ of degree $l$. 
We identify 
$\mathcal{O}_D / \mathfrak{p}_D$ with 
$k_r$. 
Let $C$ be 
the subring of $M_m(k_r)$
consisting of all upper triangular matrices. 
Let $\mathfrak{A}$ denote 
the inverse image of $C$ 
under the reduction map 
$M_m(\mathcal{O}_D) \to M_m(k_r)$. 
Then $\mathfrak{A}$ is an order in $A$. 
Let $\mathfrak{P}_{\mathfrak{A}}$ 
be the Jacobson radical of $\mathfrak{A}$. 
Note that 
$\mathfrak{P}_{\mathfrak{A}}=\varphi_{\zeta} \mathfrak{A}$ and 
the normalizer of $\mathfrak{A}$ 
in $G$ equals 
$L_{\zeta}^{\times} \mathfrak{A}^{\times}$. 
For a positive integer $i$, we set 
$U_{\mathfrak{A}}^i=1+\mathfrak{P}_{\mathfrak{A}}^i$. 
Let 
$\theta_{D,m,\eta} \colon L_{\zeta}^{\times} U_{\mathfrak{A}}^1 
 \to \mathbb{C}^{\times}$ 
be the character defined by
\begin{gather}\label{ccd}
\begin{aligned}
 \theta_{D,m,\eta} (x) &=\chi (\bar{x}) \quad 
 \textrm{for $x \in \mu_{q-1}(K)$},\\ 
 \theta_{D,m,\eta} (x) & = 
 (\psi_K \circ 
 \Trd_{A/K} )(\varphi_{\zeta}^{-1}(x-1)) \quad 
 \textrm{for $x \in U_{\mathfrak{A}}^1$},\\ 
 \theta_{D,m,\eta} (\varphi_{\zeta})& =
 (-1)^{m-1}c. 
\end{aligned}
\end{gather}
We put 
$\pi_{D,m,\eta} = 
 \mathrm{c\mathchar`-Ind}_{L_{\zeta}^{\times} U_{\mathfrak{A}}^1}^{G} 
 \theta_{D,m,\eta}$. 

\begin{definition}\label{def:sims}
We say that 
an irreducible supercuspidal representation $\pi$ of $G$ 
is simple supercuspidal if 
$\pi \simeq \pi_{D,m,\eta}$ for some 
$\eta \in \mu_{q-1}(K) \times 
 (k^{\times})^{\vee} \times \mathbb{C}^{\times}$. 
\end{definition}

\begin{lemma}\label{piepi}
The representation $\pi_{D,m,\eta}$ 
is a supercuspidal 
representation of conductor $n+1$. 
\end{lemma}
\begin{proof} 
We define a chain lattice 
$\Lambda=\{\Lambda_i\}_{i \in \mathbb{Z}}$
in $D^{\oplus m}$ by 
\[
 \Lambda_{mj+l}=
 (\mathfrak{p}_D^j)^{\oplus m-l} \oplus 
 (\mathfrak{p}_D^{j+1})^{\oplus l} 
\]
for $j \in \mathbb{Z}$ and $0 \leq l \leq m-1$. 
Then we have 
\[
 \mathfrak{A}= 
 \bigl\{ g \in A \bigm| g \Lambda_i \subset \Lambda_i \ \, 
 \textrm{for all} \ \, i \in \mathbb{Z} \bigr\} 
\] 
and 
$\varphi_{\zeta} \Lambda_i \subset \Lambda_{i+1}$ 
for 
$i \in \mathbb{Z}$. 
Hence, $\mathfrak{A}$ is a hereditary order in 
$A$ (cf.~\cite[D\'{e}finition 1.3]{SecGLmD1}). 
We consider a stratum 
$[\mathfrak{A},1,0,\varphi_{\zeta}^{-1}]$ 
of $A$ (cf.~\cite[D\'{e}finition 2.1]{SecGLmD1}).
We set 
\[
 B=\bigl\{x \in A \bigm| xz=zx\ \textrm{for all}\ z \in L_{\zeta} \bigr\}. 
\]
Then we have $B=L_{\zeta}$. 
Using this, we see that 
the critical exponent of this stratum equals 
$-1$ (cf.~\cite[2.1]{SecGLmD1}). 
Hence, this stratum is simple (cf.~\cite[D\'{e}finition 2.3]{SecGLmD1}). 
Since the simple pair 
$[0,\varphi_{\zeta}^{-1}]$ over $K$ 
is minimal in the sense of \cite[2.3.3]{SecGLmD1}, 
we have 
\begin{equation}\label{HJU}
 H^1(\varphi_{\zeta}^{-1},\mathfrak{A})
 =
 J^1(\varphi_{\zeta}^{-1},\mathfrak{A})
 =U_{\mathfrak{A}}^1 \subset 
 J(\varphi_{\zeta}^{-1},\mathfrak{A})=\mathcal{O}_{L_{\zeta}}^{\times} 
 U_{\mathfrak{A}}^1 
\end{equation}
under the notation in 
\cite[(65)]{SecGLmD1}. 
We have 
\begin{equation}\label{mato}
 \mathscr{C}(\varphi_{\zeta}^{-1},0, \mathfrak{A}) 
 =\Bigl\{ 
 \theta_{D,m,\eta}|_{U^1_{\mathfrak{A}}} 
 \Bigr\} 
\end{equation}
by \cite[Lemma 3.23]{SecGLmD1} under the notation in 
\cite[D\'{e}finition 3.45]{SecGLmD1}. 
Then, by \eqref{ccd}, \eqref{HJU} and \eqref{mato}, 
we can check that 
the pair 
\[
 (\mathrm{J},\lambda)= 
 \Bigl( \mathcal{O}_{L_{\zeta}}^{\times}U_{\mathfrak{A}}^1, 
 \theta_{D,m,\eta}|_{\mathcal{O}_{L_{\zeta}}^{\times}
 U_{\mathfrak{A}}^1} \Bigr) 
\] 
is a maximal simple type of level $>0$ for 
$G$ with respect to the 
simple stratum 
$[\mathfrak{A},1,0,\varphi_{\zeta}^{-1}]$ 
in the sense of \cite[4.1 and 5.1]{SecGLmD3}. 
Hence, $\pi_{D,m,\eta}$ is a supercuspidal representation 
by \cite[Th\'{e}or\`{e}me 5.2]{SecGLmD3}. 
We see 
$c(\pi_{D,m,\eta})=n+1$ by using 
\cite[Theorem 3.3.8]{BFgauss} (cf.~the proof of 
\cite[Proposition 2.6]{ABPS}). 
\end{proof}

\begin{proposition}\label{epiclass}
The map 
\[
 \Phi \colon \mu_{q-1}(K) \times 
 (k^{\times})^{\vee} \times \mathbb{C}^{\times} \to 
 \cA_{D,m}^{n+1}; \ 
 \eta \mapsto \pi_{D,m,\eta} 
\]
is a bijection. 
\end{proposition}
\begin{proof}
We show the injectivity. 
We take 
\[
 \eta =(\zeta,\chi,c), \ \eta'=(\zeta',\chi',c') 
 \in \mu_{q-1}(K) \times 
 (k^{\times})^{\vee} \times \mathbb{C}^{\times}
\] 
such that 
$\pi_{D,m,\eta} \simeq \pi_{D,m,\eta'}$. 
By \cite[Corollary 7.3]{SeSt6}, 
we see that 
$\zeta=\zeta'$ and $\chi=\chi'$. 
By $\pi_{D,m,\eta} \simeq \pi_{D,m,\eta'}$, 
there exists $g_0 \in G$ such that 
\begin{equation}\label{conjhom}
 \Hom_{L_{\zeta}^{\times} U_{\mathfrak{A}}^1 \cap 
 (L_{\zeta}^{\times} U_{\mathfrak{A}}^1)^{g_0}}
 (\theta_{D,m,\eta},\theta_{D,m,\eta'}^{g_0}) \neq 0. 
\end{equation}
This implies 
\[
 \Hom_{U_{\mathfrak{A}}^1 \cap 
 (U_{\mathfrak{A}}^1)^{g_0}}
 (\theta_{D,m,\eta},\theta_{D,m,\eta}^{g_0}) \neq 0, 
\]
since 
$\theta_{D,m,\eta}$ coincides with 
$\theta_{D,m,\eta'}$ on $U_{\mathfrak{A}}^1$. 
Then we have 
$g_0 \in L_{\zeta}^{\times} U_{\mathfrak{A}}^1$ 
by \cite[Proposition 2.10]{SecGLmD2}. 
Hence, we have $\eta =\eta'$ by \eqref{conjhom}. 

We show the surjectivity. 
Let $\pi \in \cA_{D,m}^{n+1}$. 
By 
\cite[Th\'{e}or\`{e}me 5.21 and Corollaire  5.22]{SeSt4}, 
we have 
$\pi \simeq 
 \mathrm{c\mathchar`-Ind}_{\bar{J}}^G \theta$ 
for 
a maximal simple type $(J,\lambda)$ 
with a simple stratum 
$[\frak{A}_0,l,0,\beta]$ 
and an extension $\theta$ of $\lambda$ 
to 
\[
 \bar{J}=\{ g \in G \mid J^g =J ,\ \lambda^g =\lambda \}. 
\]
Let $\mathfrak{P}_{\frak{A}_0}$ be the 
Jacobson radical of $\frak{A}_0$. 
By \cite[Proposition 2.6]{ABPS}, 
we have $l=1$ and 
$\mathfrak{p}_K \frak{A}_0 = \mathfrak{P}_{\frak{A}_0}^n$. 
Replacing $\frak{A}_0$ by its conjugate, 
we may assume that $\frak{A}_0 =\frak{A}$ 
(cf.~\cite[(1.5.2) Proposition (ii)]{BFgauss}). 
We put $U_K^1 =1+\mathfrak{p}_K$. 
We take $\zeta \in \mu_{q-1}(K)$ 
such that 
\begin{equation}\label{Nrdmod}
 \Nrd_{A/K} (\beta^{-1}) \equiv (-1)^{n-1} \zeta \varpi 
 \mod U_K^1. 
\end{equation}
We write $\beta^{-1} =\varphi_{\zeta} a u$, where 
$u \in U_{\frak{A}}^1$ and 
$a=(a_{ij})_{1 \leq i,j \leq m} \in \mathfrak{A}^{\times}$ 
is a diagonal matrix 
such that $a_{ii} \in \mu_{q^r -1} (K_r)$ 
for $1 \leq i \leq m$. 
We put $b= \prod_{1 \leq i \leq m} a_{ii}$. 
Then we see that 
$\varphi_{\zeta} a$ is conjugate to 
\[
 \begin{pmatrix}
 \bm{0} & I_{m-1} \\
 \varphi_{D,\zeta} b & \bm{0}
 \end{pmatrix} \in A 
\]
by an element of $\mathfrak{A}^{\times}$. 
Further, this is conjugate to $\varphi_{\zeta}$ 
by an element of $\mathfrak{A}^{\times}$, 
since we have $\Nrd_{D/K} (b)=1$ by \eqref{Nrdmod}. 
Hence, we may assume that 
$\beta =\varphi_{\zeta}^{-1}$. 

We see that $\theta$ is a character 
by the definition of a maximal simple type 
(cf.~\cite[4.1 and 5.1]{SecGLmD3}), 
\cite[2.2]{SecGLmD2} and \eqref{HJU}. 
We define 
$\chi \in (k^{\times})^{\vee}$ by 
$\chi (\bar{x}) =\theta (x)$ for $x \in \mu_{q-1}(K)$, and 
put 
$c=(-1)^{m-1} \theta (\varphi_{\zeta})$. 
Then we have $\pi=\Phi (\eta )$ for 
$\eta =(\zeta,\chi,c)$. 
\end{proof}

\begin{remark}
If $m=n$, 
Proposition \ref{epiclass} follows from 
\cite[2.1 and 2.2]{BHepi}. 
Actually, the proof of the injectivity of 
$\Phi$ is logically unnecessary in this paper: 
We need only the surjectivity of $\Phi$ 
in the case $m=n$ 
for the proof of Theorem \ref{mainthm}, 
and 
the injectivity of $\Phi$ follows from 
Theorem \ref{mainthm} and the 
injectivity of $\Phi$ in the case $m=n$. 
\end{remark}

\section{Formula for character}\label{character}
In this section, 
we give formulas for 
characters of the simple supercuspidal representations of $G$ 
at some elements. 

\begin{definition}[{cf.~\cite[1.1 Remark]{BHetJL}}]\label{qreg}
Let $g \in G$. 
Let $f_g(x) \in K[x]$ be 
the reduced characteristic polynomial of $g$ over $K$. 
\begin{enumerate}
\item 
We say that $g$ is regular if 
$f_g(x)$ is separable. 
\item 
We say that $g$ is quasi-regular if 
$f_g(x)$ has no repeated irreducible factor over $K$. 
\item
We say that $g$ is elliptic if 
the minimal polynomial of $g$ is irreducible over $K$. 
\end{enumerate}
\end{definition}

We write $G_{\mathrm{reg}}$, $G_{\mathrm{qr}}$ 
and $G_{\mathrm{qr}}^{\mathrm{ell}}$ 
for the set of the regular elements of $G$, 
the set of the quasi-regular elements of $G$ 
and the set of the elliptic quasi-regular elements of $G$ 
respectively. 

Let $\cH (G)$ be the space of locally constant 
compactly supported functions $G \to \bC$. 
We take a Haar measure $d\mu_G$ on $G$. 
We recall the following fact due to Bushnell-Henniart: 

\begin{proposition}\label{charadef} 
For an irreducible smooth representation $\pi$ of $G$, 
there is a locally constant function 
$\tr \pi \colon G_{\mathrm{qr}} \to \bC$ characterized by 
\[
 \tr \pi (f) = \int_G \tr \pi (g) f (g) d\mu_G 
\]
for all $f \in \cH (G)$ with support contained in $G_{\mathrm{qr}}$. 
\end{proposition}
\begin{proof}
This is proved in \cite[(A.11) Corollary]{BHtliftI} in the case where $G=\iGL_n (K)$. 
The same arguments work also in our situation 
(cf.~\cite[1.1 Remark and Proposition]{BHetJL}). 
\end{proof}

For an irreducible smooth representation $\pi$ of $G$, 
let $\tr \pi$ be the function in Proposition \ref{charadef}, 
which we call the character of $\pi$. 

Let $\eta=(\zeta,\chi,c) \in \mu_{q-1}(K) \times 
 (k^{\times})^{\vee} \times \mathbb{C}^{\times}$. 
\begin{lemma}\label{Mac2}
We have 
\[
 \tr \pi_{D,m,\eta}(g) = 
 \sum_{x \in G / L_{\zeta}^{\times}U_{\mathfrak{A}}^1, \, 
 x^{-1} g x \in L_{\zeta}^{\times}U_{\mathfrak{A}}^1}
 \theta_{D,m,\eta}(x^{-1}g x) 
\]
for $g \in G_{\mathrm{qr}}^{\mathrm{ell}}$. 
\end{lemma}
\begin{proof}
This follows from 
\cite[1.1 Remark and (1.2.2)]{BHetJL} (cf.~\cite[(A.14)]{BHtliftI}). 
\end{proof}
We put $U_K^1=1+\mathfrak{p}_K \subset \mathcal{O}_K^{\times}$.  
We set $g_u=\varphi_{\zeta}(1+\varphi_{\zeta}u)
\in \mathfrak{A}$ for $u \in \mathfrak{A}$. 
\begin{lemma}\label{qregphi}
Let $u \in \mathfrak{A}$. 
We set 
\[
 f_{g_u}(x)=x^n +\sum_{i=0}^{n-1} a_i(u) x^i \in K[x].
\]
Then, we have 
\[
 a_i(u) \in \mathfrak{p}_K \quad \textrm{for $1 \leq i \leq n-1$}, \quad 
 -\frac{a_0(u)}{\zeta \varpi} \in U_K^1. 
\]
Furthermore, 
$g_u$ 
is an elliptic quasi-regular element. 
\end{lemma}
\begin{proof}
We see that 
$a_i(u) \in \mathfrak{p}_K$ for $0 \leq i \leq n-1$, 
since the left multiplication 
on $\cO_D /\varpi \cO_D$ by 
$g_u$ 
acts nilpotently. 
We have 
\[
 a_0(u) =(-1)^n \mathrm{Nrd}_{A/K} (g_u) \in 
 \mathfrak{p}_K \setminus \mathfrak{p}_K^2. 
\]
Hence, the former assertion follows from
\eqref{Nrd}
%$\Nrd_{A/K}(\varphi_{\zeta})=(-1)^{n-1}\zeta \varpi$
and $\Nrd_{A/K}(1+\varphi_{\zeta} u) \in U_K^1$. 
By Eisenstein's irreducibility criterion,
$f_{g_u}(x)$ is irreducible over $K$.  
Therefore, the latter assertion follows. 
\end{proof}

\begin{lemma}\label{gLA}
Let $g \in G$ and $u \in \mathfrak{A}$. 
%Let the notation be as in Lemma \ref{qregphi}. 
Assume that 
$g^{-1} g_u g 
 \in \mathfrak{P}_{\mathfrak{A}}$. 
Then we have $g \in L_{\zeta}^{\times} \mathfrak{A}^{\times}$. 
\end{lemma}
\begin{proof}
We take $l \in \bZ$ and 
$s =(s_{i,j})_{1 \leq i,j \leq m} 
 \in \mathfrak{A} \setminus \mathfrak{P}_{\mathfrak{A}}$ 
such that 
$g=\varphi_{\zeta}^l s$. 
We put 
$u'=\varphi_{\zeta}^{-l} u \varphi_{\zeta}^l \in \mathfrak{A}$. 
Then we have 
\[
 g^{-1} g_u%\varphi_{\zeta} (1+\varphi_{\zeta} u) 
 g = 
 s^{-1} g_{u'}%\varphi_{\zeta} (1+\varphi_{\zeta} u') 
 s 
 \in \mathfrak{P}_{\mathfrak{A}}. 
\]
Hence, we have 
$g_{u'}%\varphi_{\zeta} (1+\varphi_{\zeta} u')
 s \mathfrak{A}
 \subset s \varphi_{\zeta} \mathfrak{A}$. 
This implies 
$g_{u'}%\varphi_{\zeta} (1+\varphi_{\zeta} u') 
s \mathfrak{A}
 = s \varphi_{\zeta} \mathfrak{A}$, 
since 
we have 
\[
\mathrm{Nrd}_{A/K} (\varphi_{\zeta}^{-1} s^{-1} 
g_{u'} %\varphi_{\zeta} (1+\varphi_{\zeta} u') 
s) \in \cO_K^{\times}.
\] 
Therefore we have 
$\varphi_{\zeta} s \mathfrak{A}
 = s \varphi_{\zeta} \mathfrak{A}$ in 
$\mathfrak{P}_{\mathfrak{A}} / \mathfrak{P}_{\mathfrak{A}}^2$. 
This implies that 
$s_{i,i} \in \cO_D^{\times}$ if and only if 
$s_{i+1,i+1} \in \cO_D^{\times}$ 
for each $1 \leq i \leq m-1$. 
Then we have $s_{i,i} \in \cO_D^{\times}$ for 
all $1 \leq i \leq m$, since 
$s \notin \mathfrak{P}_{\mathfrak{A}}$. 
Hence, we have $s \in \mathfrak{A}^{\times}$. 
This shows the claim. 
\end{proof}

We put $n_q=(n,q-1)$. 
For $a \in k$, we put 
\[
 G_n (\chi, \psi, a) = 
 \sum_{x \in \mu_{n_q}(k)} 
 \chi (x) \psi (ax). 
\]
Note that $G_{q-1} (\chi, \psi, 1)$ is a usual Gauss sum, 
for which we write $G(\chi, \psi)$. 

\begin{proposition}\label{trGa}
For $u \in \mathfrak{A}$, we have 
\[
 \tr \pi_{D,m,\eta} \bigl( g_u%\varphi_{\zeta}(1+\varphi_{\zeta} u) 
 \bigr) = 
 (-1)^{m-1} c\, G_n \bigl( \chi, \psi, \overline{\Trd_{A/K}(u)} \bigr). 
\]
\end{proposition}
\begin{proof}
For $\lambda \in \mu_{q-1}(K)$, let 
\[
 g(\lambda)  =
 \mathrm{diag} (1,\lambda,\cdots,\lambda^{m-1}) \in G. 
\]
We put 
\[
 H =\bigl\{d g(\lambda) \in 
 G \bigm| d \in 
 \mu_{q^r-1}(K_r), \ 
 \lambda \in \mu_{q-1}(K), \ 
 d^{q^s -1}=\lambda^m \bigr\}. 
\]
Note that 
$\lambda$ in the definition of $H$ 
automatically belongs to $\mu_{n_q}(K)$. 
Then, we can check that 
$H$ is equal to 
\[
 \bigl\{ x= \mathrm{diag} (d_1,\cdots,d_m ) \in G 
 \bigm| 
 x^{-1} g_u%\varphi_{\zeta}(1+\varphi_{\zeta} u)
  x \in L_{\zeta}^{\times} U_{\mathfrak{A}}^1,\ 
 d_i \in \mu_{q^r-1}(K_r) \ 
 \textrm{for all}\ i 
 \bigr\}. 
\]
Hence, we see that 
\[
 \bigl\{x \in G \bigm| 
 x^{-1} g_u%\varphi_{\zeta}(1+\varphi_{\zeta} u) 
 x \in L_{\zeta}^{\times} U_{\mathfrak{A}}^1 
 \bigr\} =H L_{\zeta}^{\times} U_{\mathfrak{A}}^1 
\]
by Lemma \ref{gLA}. 
We have the isomorphism 
\[
 H L_{\zeta}^{\times} U_{\mathfrak{A}}^1/
 L_{\zeta}^{\times}U_{\mathfrak{A}}^1 \xrightarrow{\sim}
 \mu_{n_q}(k);\ 
 d g(\lambda) \mapsto \bar{\lambda}. 
\]
Hence, the claim follows from 
Lemma \ref{Mac2} and Lemma \ref{qregphi}. 
\end{proof}

\begin{lemma}\label{qreg1phi}
The element 
$1+\varphi_{\zeta}$ 
is elliptic quasi-regular. 
\end{lemma}
\begin{proof}
The element $\varphi_{\zeta}$ is 
elliptic quasi-regular by Lemma \ref{qregphi}. 
Hence, the claim follows. 
\end{proof}

\begin{lemma}\label{gLA2}
Let $g \in G$ and $\lambda \in \mu_{q-1}(K)$. 
Assume that 
$g^{-1} (1+\varphi_{\zeta \lambda}) g 
 \in L_{\zeta}^{\times} U_{\mathfrak{A}}^1$. 
Then we have $g \in L_{\zeta}^{\times} \mathfrak{A}^{\times}$. 
\end{lemma}
\begin{proof}
By the assumption, we have 
$g^{-1} (1+\varphi_{\zeta \lambda}) g =\lambda_0 (1+ v)$ 
with $\lambda_0 \in \mu_{q-1} (K)$ and 
$v \in \mathfrak{P}_{\mathfrak{A}}$. 
Then we have 
\[
 \lambda_0 = \lim_{l \to \infty} 
 \bigl( \lambda_0 (1+ v) \bigr)^{q^l}
 =\lim_{l \to \infty} 
 \bigl( g^{-1} (1+\varphi_{\zeta \lambda}) g \bigr)^{q^l} =1. 
\]
Hence we have 
$g^{-1} \varphi_{\zeta \lambda} g =v \in \mathfrak{P}_{\mathfrak{A}}$. 
Therefore, the claim follows from Lemma \ref{gLA} and 
$L_{\zeta}^{\times} \mathfrak{A}^{\times} = 
 L_{\zeta \lambda}^{\times} \mathfrak{A}^{\times}$. 
\end{proof}

For a finite field extension $k'$ of $k$, 
an additive character $\psi' \in k'^{\vee}$, 
a positive integer $l$ and 
$a \in k'^{\times}$, 
we put 
\[
 K_{l,a} (\psi') = 
 \sum_{\zeta_1 \cdots \zeta_l =a,\, \zeta_i \in k'^{\times}} 
 \psi' \biggl( \sum_{i=1}^{l} \zeta_i \biggr). 
\]
This is a generalized Kloosterman sum 
(cf.~\cite[Sommes trig.~7.1]{DeCE}). 

\begin{proposition}\label{trKl}
For $\lambda \in \mu_{q-1}(K)$, 
we have 
\begin{equation}\label{d1}
 \tr \pi_{D,m,\eta}(1+\varphi_{\zeta \lambda})
 = (-1)^{n-m} K_{n,\bar{\lambda}} (\psi). 
\end{equation}
\end{proposition}
\begin{proof}
Let $\lambda \in \mu_{q-1}(K)$. 
We have 
\begin{equation*}
 \tr \pi_{D,m,\eta} 
 (1+\varphi_{\zeta \lambda}) 
 = \sum_{g \in L_{\zeta}^{\times} \mathfrak{A}^{\times}/
 L_{\zeta}^{\times}U_{\mathfrak{A}}^1} 
 \theta_{D,m,\eta}(g^{-1}(1+\varphi_{\zeta \lambda})g) 
\end{equation*}
by Lemma \ref{Mac2}, Lemma \ref{qreg1phi} and Lemma \ref{gLA2}. 
We can check 
that 
\[
 \tr_{A/D} 
 (\varphi_{\zeta}^{-1}g^{-1}\varphi_{\zeta \lambda}g)
 \equiv 
 \varphi_{D,\zeta}^{-1}a_{m,m}^{-1} \varphi_{D, \zeta \lambda} 
 a_{1,1}+\sum_{i=1}^{m-1} a_{i,i}^{-1}a_{i+1,i+1} 
 \mod \mathfrak{p}_D 
\]
for 
$g=(a_{i,j})_{1 \leq i,j \leq m} \in \mathfrak{A}^{\times}$. 
For $d \in \mathcal{O}_D$, 
let $\bar{d}$ denote the image of $d$
by $\mathcal{O}_D \to k_r$. 
We have the bijection 
\begin{align*}
 L_{\zeta}^{\times}\mathfrak{A}^{\times}/
 L_{\zeta}^{\times}U_{\mathfrak{A}}^1 &\xrightarrow{\sim} 
 \bigl\{ (\zeta_1, \cdots , \zeta_m) \in (k_r^{\times})^{\oplus m} 
 \bigm| \Nr_{k_r/k} (\zeta_1 \cdots \zeta_m )=\bar{\lambda} \bigr\} \\
 (a_{i,j})_{1 \leq i,j \leq m} &\mapsto 
 \Bigl( ( \bar{a}_{i,i}^{-1} \bar{a}_{i+1,i+1} )_{1 \leq i \leq m-1}, 
 \overline{ \varphi_{D,\zeta}^{-1}a_{m,m}^{-1} \varphi_{D, \zeta \lambda} a_{1,1} } \Bigr), 
\end{align*}
where $(a_{i,j})_{1 \leq i,j \leq m} \in \mathfrak{A}^{\times}$. 
Hence, we have 
\[
 \tr \pi_{D,m,\eta} 
 (1+\varphi_{\zeta \lambda}) 
 = 
 \sum_{\Nr_{k_r/k}(x)=\bar{\lambda}} K_{m,x} (\psi \circ \Tr_{k_r/k}). 
\]
Further, we have 
\[
  \sum_{\Nr_{k_r/k}(x)=\bar{\lambda}} K_{m,x} (\psi \circ \Tr_{k_r/k})
 =(-1)^{m-1} \sum_{\Nr_{k_n/k}(y)=\bar{\lambda}} 
 \psi \circ \Tr_{k_n/k}(y)=(-1)^{n-m} 
 K_{n,\bar{\lambda}}(\psi) 
\]
by \cite[(7.2.5)]{DeCE}. 
Thus, we have proved the claim. 
\end{proof}

\section{Description of LJLC}\label{DesLJLC}
In this section, we give an explicit description of 
the local Jacquet-Langlands correspondence (LJLC). 

First, we recall the statement of the LJLC. 
Let $\mathcal{A}^{\square}_{D,m}$ 
denote the set of equivalent classes of 
irreducible essentially square-integrable 
smooth representations of 
$\textit{GL}_m(D)$. 
We put $G'=\iGL_n (K)$. 
We say that 
$g \in G_{\mathrm{qr}}$ and 
$g' \in G'_{\mathrm{qr}}$ correspond 
if $f_{g} (x)=f_{g'} (x)$. 

\begin{theorem}[{\cite[Th\'{e}or\`{e}me 5.1]{BadJLn} and \cite[Th\'{e}or\`{e}me principal]{DKV}}]\label{LJLCthm}
There exists a bijection 
\[
 \mathrm{JL}_{D,m} \colon \mathcal{A}^{\square}_{D,m} 
 \to \mathcal{A}^{\square}_{K,n}, 
\]
which is uniquely characterized by 
the character relation 
\[
 \tr \pi (g)=(-1)^{n-m} \tr 
 \mathrm{JL}_{D,m}(\pi ) (g')  
\]
for 
$\pi \in \mathcal{A}^{\square}_{D,m}$ and 
any corresponding 
$g \in G_{\mathrm{reg}}$ and $g' \in G'_{\mathrm{reg}}$. 
\end{theorem}

The bijection $\mathrm{JL}_{D,m}$ in 
Theorem \ref{LJLCthm} is called 
the local Jacquet-Langlands correspondence. 

\begin{corollary}\label{jl}
The bijection $\mathrm{JL}_{D,m}$ 
satisfies 
the character relation 
\[
 \tr \pi (g)=(-1)^{n-m} \tr 
 \mathrm{JL}_{D,m}(\pi ) (g') 
\]
for $\pi \in \mathcal{A}^{\square}_{D,m}$ and 
any corresponding 
$g \in G_{\mathrm{qr}}$ and 
$g' \in G'_{\mathrm{qr}}$. 
\end{corollary}
\begin{proof}
The complement of $G^{\mathrm{reg}}$ in $G$ is a Zariski closed subset of
$G$ with strictly smaller dimension than $G$. 
This implies that $G^{\mathrm{reg}}$ is dense in $G$ 
in the $\varpi$-adic topology, 
hence it is dense also in $G^{\mathrm{qr}}$.
Therefore, the claim follows from 
Theorem \ref{LJLCthm}, 
since $\tr \pi$ and $\tr \mathrm{JL}_{D,m}(\pi)$ 
for $\pi \in \mathcal{A}^{\square}_{D,m}$ 
are locally constant functions on 
$G_{\mathrm{qr}}$ and 
$G'_{\mathrm{qr}}$ respectively. 
\begin{comment}
The set of regular elements is 
dense in the set of quasi-regular elements, 
since any inseparable polynomial over $K$ 
can be 
$\mathfrak{p}_K$-adically approximated 
arbitrarily by a separable polynomial over $K$. 
\end{comment}
\end{proof}

\begin{lemma}\label{fou}
The function 
\[
 G_{n,\chi,\psi} \colon k \to \mathbb{C};\ 
 a \mapsto 
 G_n (\chi,\psi,a). 
\]
is not equal to 
the zero function on $k$. 
\end{lemma}
\begin{proof}
Let $F_{n,\chi}$ be the function on $k$ defined by 
\[
 F_{n,\chi}(x)=
 \begin{cases}
 \chi(x) & \textrm{if}\ x \in \mu_{n_q}(k),\\ 
 0 & \textrm{if}\ x \notin \mu_{n_q}(k).
 \end{cases} 
\]
Then 
$G_{n,\chi,\psi}$ is regarded as 
the Fourier transformation of $F_{n,\chi}$ with respect to $\psi$. 
Hence, the claim follows from the Fourier inversion formula. 
\end{proof}

\begin{lemma}\label{cc}
For any $a' \in k^{\times} \setminus \{ 1 \}$, 
there is $a \in k^{\times}$ such that 
$K_{n,a}(\psi) \neq K_{n,a a'}(\psi)$. 
\end{lemma}
\begin{proof}
Let $a' \in k^{\times} \setminus \{ 1 \}$. 
We take $\chi' \in (k^{\times})^{\vee}$ 
such that 
$\chi' (a' ) \neq 1$. 
Then we have 
\begin{align*}
 \sum_{a \in k^{\times}} \chi' (aa') 
 \bigl( K_{n,a}(\psi) - K_{n,a a'}(\psi) \bigr) 
 &=(\chi' (a') -1) \sum_{a \in k^{\times}} 
 \chi' (a) K_{n,a}(\psi) \\ 
 &= 
 (\chi' (a') -1) G(\chi' ,\psi )^n \neq 0 
\end{align*}
by \cite[(7.1.6)]{DeCE}. 
Hence, the claim follows. 
\end{proof}

\begin{theorem}\label{mainthm}
For $\eta=(\zeta,\chi,c) \in 
\mu_{q-1}(K) \times (k^{\times})^{\vee} \times \mathbb{C}^{\times}$, we have 
$\mathrm{JL}_{D,m}(\pi_{D,m,\eta})=\pi_{K,n,\eta}$. 
\end{theorem}
\begin{proof}
Let $\mathrm{LJ}_{D,m}$ denote the inverse of 
$\mathrm{JL}_{D,m}$. 
The map $\mathrm{LJ}_{D,m}$ preserves 
conductors and supercuspidality by \cite[Theorem 2.1(e), (i)]{ABPS}. 
Let $\eta'=(\zeta',\chi',c') \in \mu_{q-1}(K) \times (k^{\times})^{\vee} \times \mathbb{C}^{\times}$. 
By Proposition \ref{epiclass}, 
there exists $\eta=(\zeta,\chi,c) \in 
 \mu_{q-1}(K) \times (k^{\times})^{\vee} \times \mathbb{C}^{\times}$ 
such that 
$\mathrm{LJ}_{D,m}(\pi_{K,n,\eta'})=\pi_{D,m,\eta}$. 

\begin{comment}
We set $\pi=\mathrm{JL}_{D,m}(\pi_{D,m,\eta})$. 
First, we show that 
$\pi$ is supercuspidal. 
Let $\delta(\pi_{D,m,\eta})$ denote 
the parametric degree of $\pi_{D,m,\eta}$
in \cite[2.7 Definition]{BHetJL}. 
Then $\delta(\pi_{D,m,\eta})$ is 
divisible by $n=[L_{\zeta}:K]$ by definition.  
We have 
$\delta(\pi_{D,m,\eta}) \mid n$ by \cite[(2.7.3)]{BHetJL}. Hence, 
we have $\delta(\pi_{D,m,\eta})=n$. 
Since $\mathrm{JL}_{D,m}$ preserves 
parametric degrees by \cite[2.8 Corollary 1]{BHetJL}, 
the parametric degree of 
$\pi$ equals $n$.
Then, we know that $\pi$ is supercuspidal 
by \cite[1.1 Lemma]{BHpara}. 
Since $\mathrm{JL}_{D,m}$ 
preserves conductors, 
$\pi$ is simple supercuspidal. 
Hence, by Proposition \ref{epiclass}, we can take 
$\eta'=(\zeta',\chi',c') \in 
 \mu_{q-1}(K) \times (k^{\times})^{\vee} \times \mathbb{C}^{\times}$ 
such that 
$\pi=\pi_{K,n,\eta'}$. 
\end{comment}

We have to show $\eta=\eta'$. 
We have $\chi=\chi'$, 
because $\mathrm{JL}_{D,m}$ preserves central characters.  
We write $\mathfrak{A}'$ and 
$\varphi'_{\zeta}$ for $\mathfrak{A}$ and 
$\varphi_{\zeta}$ respectively in the case $m=n$. 
By Proposition \ref{trKl} and Corollary \ref{jl}, 
we have 
\[
 K_{n,\bar{\lambda}} (\psi)= 
 (-1)^{n-m} \tr \pi_{D,m,\eta} (1+\varphi_{\zeta \lambda})
 = \tr \pi_{K,n,\eta'} (1+\varphi'_{\zeta \lambda})
 = K_{n,\overline{\lambda \zeta \zeta'^{-1}}} (\psi) 
\]
for $\lambda \in \mu_{q-1} (K)$. 
Hence, we have $\zeta=\zeta'$ 
by Lemma \ref{cc}. 

Finally, we show $c=c'$. 
By Lemma \ref{fou}, 
we can take $a \in k$
such that $G_n(\chi,\psi, a) \neq 0$. 
%Since $\Tr_{k_r/k}$ is surjective,  
We take $\xi_0 \in k_r$ 
such that $\Tr_{k_r/k}(\xi_0)=a$. 
Let $\xi \in \mathcal{O}_D$ be a lifting of
$\xi_0$.
We set 
\[
u_{\xi}=\mathrm{diag}(\xi,0,\cdots,0) \in \mathfrak{A}. 
\]
We simply write $g_{\xi}$ for 
$g_{u_{\xi}}$. 
Let $\{a_i(u_{\xi})\}_{0 \leq i \leq n-1}$ be as in Lemma \ref{qregphi}.  
By Lemma \ref{qregphi}, the elements 
\[
 \alpha_i=-\frac{a_i(u_{\xi})}{\zeta \varpi} \quad 
 \textrm{for $1 \leq i \leq n-1$}, \quad 
 \alpha_n=-\frac{\frac{a_0(u_{\xi})}{\zeta\varpi}+1}{\zeta \varpi} 
\]
are contained in $\mathcal{O}_K$. 
We set 
\begin{align*}
 u_{\alpha}=\sum_{i=1}^n 
 \varphi'^{i-1}_{\zeta} \mathrm{diag}(\alpha_i,0,\cdots,0) \in 
 \mathfrak{A}', \quad 
 g_{\alpha}=\varphi'_{\zeta}(1+\varphi'_{\zeta} u_{\alpha}) \in G'. 
\end{align*}
Then we have 
\begin{equation*}%\label{aaa}
 f_{g_{\alpha}}(x)=x^n-\sum_{i=1}^{n-1}
 \zeta \varpi \alpha_i x^i-\zeta \varpi(1+\alpha_n \zeta \varpi) 
 =f_{g_{\xi}}(x). 
\end{equation*}
Hence, 
$g_{\xi} \in G_{\mathrm{qr}}$ and 
$g_{\alpha} \in G'_{\mathrm{qr}}$ correspond. 
We have $\overline{\tr (u_{\alpha})}
=\overline{\Trd_{A/K}(u_{\xi})}=a$, 
since 
\[
 \tr (u_{\alpha}) \equiv 
 \alpha_1 = (-1)^n 
 \frac{\Trd_{A/K} (g_{\xi}^{-1}) \Nrd_{A/K} (g_{\xi})}{\zeta \varpi} 
 \equiv 
 \Trd_{D/K} (\xi) \mod \fp_K. 
\]
Therefore, 
we have 
\begin{align*}
 c =(-1)^{m-1} \frac{\tr \pi_{D,m,\eta}(g_{\xi})}{G_n(\chi,\psi,a)}= 
 (-1)^{n-1} \frac{\tr \pi_{K,n,\eta'}(g_{\alpha})}{G_n(\chi,\psi,a)} 
 =c' 
\end{align*}
by Proposition \ref{trGa} and Corollary \ref{jl}. 
\end{proof}

%\begin{remark}It will be possible to show Theorem \ref{mainthm} by calculating epsilon factors of character twists of simple supercuspidal representations. However, we prefer an approach by the character relation, since it is a standard characterization of the LJLC. \end{remark}

\section{Another proof of Theorem \ref{mainthm}}\label{An} 
Let $\eta=(\zeta,\chi,c) \in \mu_{q-1}(K)
\times (k^{\times})^{\vee} \times 
\mathbb{C}^{\times}$. 
%Let $\epsilon(\pi_{D,m,\eta},s,\psi_K)$ denote the Godement-Jacquet local constant of $\pi_{D,m,\eta}$ in \cite{GJzf}. 
We simply write 
$\epsilon (\pi_{D,m,\eta},\psi_K)$
for the value of $\epsilon(\pi_{D,m,\eta},s,\psi_K)$
 at $1/2$.
 For any smooth character $\xi \colon 
 K^{\times} \to \mathbb{C}^{\times}$, 
 let $\xi_A$ denote the composite $\xi \circ 
 \Nrd_{A/K}$.

The following lemma is a special case of a formula which 
represents Godement-Jacquet local constants 
with respect to 
non-abelian congruence Gauss sums. 
Such a formula is studied in \cite{B} and \cite{BFgauss}. 
If $m=n$, this lemma is just \cite[2.2 Lemma]{BHepi}. 
\begin{lemma}\label{ess}
{\rm 1}.\ We have 
 \[
 \epsilon(\pi_{D,m,\eta},\psi_K)=(-1)^{n-1} c. 
 \]
 {\rm 2}.\ For any tamely ramified character 
$\xi$ of $K^{\times}$, we have 
\[
\epsilon(\pi_{D,m,\eta} \otimes \xi_A,\psi_K)
=\xi \bigl( (-1)^{n-1}\zeta \varpi \bigr) 
\epsilon (\pi_{D,m,\eta},\psi_K). 
\]
\end{lemma}
\begin{proof}
Let $\xi$ be a tamely ramified character of $K^{\times}$. 
Assume that the characteristic of $K$ 
equals zero. 
We set 
\[
\rho=\mathrm{Ind}_{L_{\zeta}^{\times} U_{\mathfrak{A}}^1}^{L_{\zeta}^{\times}\mathfrak{A}^{\times}} (\theta_{D,m,\eta} \otimes \xi_A). 
\]
We see that $\rho$ is an irreducible admissible 
representation of $L_{\zeta}^{\times}\mathfrak{A}^{\times}$ by Lemma \ref{piepi}, and non-degenerate 
in the sense of \cite[p.\ 228]{BFgauss} by \eqref{ccd}. 

Let $\psi_K^0$ be the standard continuous additive character of $K$ 
defined in \cite[2.1]{BFgauss}. 
In \cite[(2.3.6)]{BFgauss}, the non-abelian congruence 
Gauss sum $\tau(\rho)$ is defined for $\rho$ and 
$\psi_K^0$. 
However, it can be defined for any non-trivial additive character 
$\psi_K'$ of $K$ as in \cite[(3.4)]{B} 
(cf.~\cite[Remarks after Theorem (3.7)]{B}), for which 
we write $\tau(\rho,\psi_K')$. 

We take an element $a \in K^{\times}$
such that $\psi_K(x)=\psi_K^0(ax)$ for $x \in K$. 
Let $\omega_{\rho^{\vee}}$ denote the central 
character of $\rho^{\vee}$. 
By \cite[(3.4)]{B} and \cite[(3.3.5)]{GJzf}, we have formulas 
\begin{gather}\label{acx}
\begin{aligned}
 \tau(\rho^{\vee},\psi_K)&=\omega_{\rho^{\vee}}(a^{-1})
 \tau(\rho^{\vee},\psi_K^0 )=(\theta_{D,m,\eta} \otimes \xi_A)(a) 
 \tau(\rho^{\vee},\psi_K^0 ), \\ 
 \epsilon(\pi_{D,m,\eta}\otimes \xi_A,\psi_K )&= 
 (\theta_{D,m,\eta} \otimes \xi_A)(a) 
 \epsilon(\pi_{D,m,\eta}\otimes \xi_A,\psi_K^0) 
\end{aligned}
\end{gather}
respectively. 
By these formulas and \cite[Theorem (3.3.8)(iv)]{BFgauss}, 
we obtain 
\begin{equation}\label{f1f}
 \epsilon(\pi_{D,m,\eta} \otimes \xi_A,\psi_K)=
 (-1)^{n-m}\tau(\rho^{\vee},\psi_K)\mathrm{N}\mathfrak{f}(\rho)^{-1/2}. 
\end{equation}
We put $\psi_A =\psi_K \circ \Trd_{A/K}$. Then we have 
\begin{equation}\label{tauN}
\tau(\rho^{\vee},\psi_K)=
\mathrm{N}\mathfrak{f}(\rho)^{1/2}
(\theta_{D,m,\eta} \otimes \xi_A)^{\vee}(\varphi_{\zeta}^{-1}) \psi_A(\varphi_{\zeta}^{-1}) 
\end{equation}
by \eqref{acx} and \cite[(2.7.4)]{BFgauss}. 
We obtain 
\begin{align*}
\epsilon(\pi_{D,m,\eta} \otimes \xi_A,\psi_K)
&=(-1)^{n-m} 
(\theta_{D,m,\eta} \otimes \xi_A)^{\vee}(\varphi_{\zeta}^{-1}) \psi_A(\varphi_{\zeta}^{-1}) \\
&=(-1)^{n-1} 
\xi \bigl( (-1)^{n-1}\zeta \varpi \bigr) c 
\end{align*}
by \eqref{Nrd}, \eqref{ccd}, \eqref{f1f} and \eqref{tauN}. 
The same arguments work 
also in the equal characteristic case 
(cf.~\cite[Remarks after Theorem (3.7)]{B}). 
Hence, we obtain the claims.  
\end{proof}

\begin{proof}[Another proof of Theorem \ref{mainthm}]
Let $\eta'=(\zeta',\chi',c') \in \mu_{q-1}(K) \times (k^{\times})^{\vee} \times \mathbb{C}^{\times}$. 
As in the first proof of Theorem \ref{mainthm}, 
we see that there exists a triple 
$\eta=(\zeta,\chi',c) \in \mu_{q-1}(K) \times (k^{\times})^{\vee} \times \mathbb{C}^{\times}$ such that $\pi_{D,m,\eta}
\simeq \mathrm{LJ}_{D,m}(\pi_{K,n,\eta'})$. 
We prove $\zeta=\zeta'$. 
Let $\xi$ be any tamely ramified character 
of $K^{\times}$. 
We write $A'$ for $A$ in the case $m=n$. 
Since the 
$\mathrm{JL}_{D,m}$ preserves character twists,  
we have 
\begin{gather}\label{bfg0}
\begin{aligned}
\epsilon (\pi_{D,m,\eta} \otimes \xi_A,\psi_K)=
\epsilon (\pi_{K,n,\eta'} \otimes \xi_{A'},\psi_K)
\end{aligned}
\end{gather}
by \cite[Theorem 2.2(i)]{ABPS} (cf.~\cite[B.j.1]{DKV}). 
We have $\epsilon(\pi_{K,n,\eta'},\psi_K) \neq 0$
by Lemma \ref{ess}.1.  
By Lemma \ref{ess}.2 and \eqref{bfg0}, 
we have 
\begin{align*}
 \xi\bigl( (-1)^{n-1}\zeta \varpi \bigr) 
 =\frac{\epsilon (\pi_{D,m,\eta} \otimes \xi_A,\psi_K)}{\epsilon (\pi_{D,m,\eta},\psi_K)} 
 =\frac{\epsilon (\pi_{K,n,\eta'} \otimes \xi_{A'},\psi_K)}{\epsilon 
(\pi_{K,n,\eta'},\psi_K)} 
 =\xi \bigl( (-1)^{n-1}\zeta' \varpi \bigr). 
\end{align*}
Hence, we obtain $\xi(\zeta \zeta'^{-1})=1$. 
This implies $\zeta=\zeta'$. 
By Lemma \ref{ess}.1 and \eqref{bfg0}, 
we have  
 \[
c= (-1)^{n-1} \epsilon (\pi_{D,m,\eta},\psi_K)=
(-1)^{n-1}\epsilon (\pi_{K,n,\eta'},\psi_K)
= c'. 
\]
Therefore we have $\eta=\eta'$. 
\end{proof}

\section{Invariance of endo-class}\label{Invec}
We show that endo-classes 
for the simple supercuspidal representations 
are invariant under the LJLC. 

Let 
$\cE (K)$ be the 
set of endo-classes of ps-characters over $K$ 
in the sense of 
\cite[Definition 1.10]{BSS5}. 
Then we have a map 
\[
 \Theta_G \colon 
 \cA^{\square}_{D,m} \to \cE (K) 
\]
constructed in \cite[9.3]{BSS5}. 
Broussous-S\'{e}cherre-Stevens conjecture that 
\[
 \Theta_{G} (\pi) = 
 \Theta_{G'} \bigl( 
 \mathrm{JL}_{D,m} (\pi) \bigr) 
\]
for $\pi \in \cA^{\square}_{D,m}$ 
in \cite[Conjecture 9.5]{BSS5}. 
The following proposition verifies the conjecture 
for the simple supercuspidal representations. 

\begin{proposition}
For any simple supercuspidal representation $\pi$ of $G$, 
we have 
\[
 \Theta_{G} (\pi) = 
 \Theta_{G'} \bigl( 
 \mathrm{JL}_{D,m} (\pi) \bigr). 
\]
\end{proposition}
\begin{proof}
Let $\pi$ be a simple supercuspidal representation of $G$. 
We take 
$\eta=(\zeta,\chi,c) \in 
 \mu_{q-1}(K) \times (k^{\times})^{\vee} \times \mathbb{C}^{\times}$ 
such that 
$\pi =\pi_{D,m,\eta}$. 
Then 
$\Theta_{G} (\pi_{D,m,\eta})$ is the endo-class 
of the ps-character defined by 
\[
 \bigl( [\mathfrak{A},1,0,\varphi_{\zeta}^{-1} ],
 \theta_{D,m,\eta}|_{U^1_{\mathfrak{A}}} \bigr). 
\]
We have 
$\mathrm{JL}_{D,m} (\pi_{D,m,\eta}) =\pi_{K,n,\eta}$ 
by Theorem \ref{mainthm}. 
We define 
$\mathfrak{A}'$ and $\varphi_{\zeta}'$ 
for $M_n (K)$ similarly as 
$\mathfrak{A}$ and $\varphi_{\zeta}$ for $M_m (D)$. 
Then $\Theta_{G} (\pi_{K,n,\eta})$ is the endo-class 
of the ps-character defined by 
\[
 \bigl( [\mathfrak{A}',1,0,\varphi_{\zeta}'^{-1} ],
 \theta_{K,n,\eta}|_{U^1_{\mathfrak{A}'}} \bigr). 
\]
Let $L_{\zeta,0}$ be a finite extension of $K$ 
generated by an element $\varphi_{\zeta,0}$ such that 
$\varphi_{\zeta,0}^n =\zeta \varpi$. 
Let 
$\mathscr{C}_K (\varphi_{\zeta,0}^{-1},0)$ 
be the set of simple characters corresponding to 
the simple pair $[\varphi_{\zeta,0}^{-1},0]$ over $K$ 
defined in \cite[p.~371]{SecGLmD1}. 
Then we have bijective transfer maps 
\begin{align*}
 \tau_{\mathfrak{A},0,\varphi_{\zeta,0}^{-1}} &\colon 
 \mathscr{C}_K (\varphi_{\zeta,0}^{-1},0) \to 
 \mathscr{C} (\varphi_{\zeta}^{-1},0,\mathfrak{A} ), \\ 
 \tau_{\mathfrak{A}',0,\varphi_{\zeta,0}^{-1}} &\colon 
 \mathscr{C}_K (\varphi_{\zeta,0}^{-1},0) \to 
 \mathscr{C} (\varphi_{\zeta}'^{-1},0,\mathfrak{A}' ) . 
\end{align*}
by \cite[Th\'{e}or\`{e}me 3.53]{SecGLmD1}. 
We have 
\[
 \tau_{\mathfrak{A},0,\varphi_{\zeta,0}^{-1}}^{-1} 
 \bigl( \theta_{D,m,\eta}|_{U^1_{\mathfrak{A}}} \bigr) 
 =
 \tau_{\mathfrak{A}',0,\varphi_{\zeta,0}^{-1}}^{-1} 
 \bigl( \theta_{K,n,\eta}|_{U^1_{\mathfrak{A}'}} \bigr) 
\] 
by \eqref{mato}. 
Hence, we have the claim. 
\end{proof}

\noindent
Naoki Imai\\ 
Graduate School of Mathematical Sciences, 
The University of Tokyo, 3-8-1 Komaba, Meguro-ku, 
Tokyo, 153-8914, Japan\\ 
naoki@ms.u-tokyo.ac.jp\\ 

\noindent
Takahiro Tsushima\\ 
Department of Mathematics and Informatics, 
Faculty of Science, Chiba University, 
1-33 Yayoi-cho, Inage, Chiba, 263-8522, Japan\\
tsushima@math.s.chiba-u.ac.jp

\end{document}